\documentclass[11pt,reqno]{amsart}
\usepackage{amscd,amssymb,amsmath,amsthm}
\usepackage[english]{babel}
\usepackage[arrow,matrix]{xy}
\usepackage{graphicx}
\usepackage{cite}
\usepackage{bbm}
\usepackage{color}
\usepackage{MnSymbol} 
\usepackage{enumerate}
\usepackage{enumitem}

\newtheorem{thm}[subsection]{Theorem}

\newtheorem{pro}[subsection]{Proposition}

\theoremstyle{definition}

\newtheorem{defn}[subsection]{Definition}

\numberwithin{equation}{section} 

\newcommand{\F}{{\mathcal F}}

\newcommand{\E}{{\mathcal E}}
\newcommand{\V}{{\mathcal V}}

\newcommand{\bea}{\begin{eqnarray}}
\newcommand{\eea}{\end{eqnarray}}

\newcommand{\N}{\mathbb N}
\newcommand{\Z}{\mathbb{Z}}

\newcommand{\R}{\mathbb{R}}

\newcommand{\PP}{\mathbb{P}}

\def \> {\Rightarrow}
\def \0 {\emptyset}

\newcommand{\W}{\mathbb P} 

\begin{document}
\title[On the Random Dynamics Of Volterra Quadratic Operators]{On the Random Dynamics Of Volterra Quadratic Operators}

\author{U.U. Jamilov, M. Scheutzow, M. Wilke-Berenguer, }

\address{M.\ Scheutzow and M. Wilke-Berenguer \\ Institut f\"ur Mathematik, MA 7-5, Fakult\"at II,
        Technische Universit\"at Berlin, Stra\ss e des 17.~Juni 136, 10623 Berlin, FRG;}
\email {ms@math.tu-berlin.de, wilkeber@math.tu-berlin.de}
  \address{U.\ U.\ Jamilov\\ Institute of mathematics at the National University of Uzbekistan,
29, Do'rmon Yo'li str., 100125, Tashkent, Uzbekistan.}
\email {jamilovu@yandex.ru}

\date{\today}

  \maketitle

\begin{abstract}
     We consider random dynamical systems generated by a special class of Volterra quadratic stochastic operators on the simplex $S^{m-1}$.
We prove that in contrast to the deterministic set-up the trajectories of the random dynamical system almost surely converge to one of the vertices of the simplex $S^{m-1}$ implying the survival of only one species. We also show that the minimal random point attractor of the system equals the set of all vertices. The convergence proof relies on a martingale-type limit theorem which we prove in the appendix.
\end{abstract}

\vskip 0.5 truecm
\maketitle


\vskip 0.5 truecm


\section{Introduction}

The concept of a quadratic stochastic operator (QSO) and its application in a biological context were first established by S.N.~Bernstein in \cite{Br}. Since then the theory has been further deepend as they frequently occur in mathematical models of genetics, where QSOs serve as a tool for the study of dynamical properties and modeling, see  \cite{NN0}--\cite{GMR}, \cite{K1}, \cite{K2}, \cite{Lyu1}, \cite{RJ1}, \cite{RJ2}, \cite{Ul}--\cite{Zakh}. While they were originally introduced as ``evolutionary operators'' to describe the dynamics of gene frequencies for given laws of heredity in mathematical population genetics, QSOs and the dynamical systems they describe have become interesting objects of study in their own right from a purely mathematical point of view (see\cite{Lyu1} for a comprehensive account).

In the description of the genetic evolution of large populations QSOs arise as follows: Consider a population with $m \in \mathbb N$ different genetic types, where every individual in this population belongs to precisely one of the species $\lsem m \rsem :=\{1, 2, \dots, m\}$. Let $\mathbf{x}^0 = (x^0_1,...,x^0_m)$ be a probability distribution on $\lsem m \rsem$ describing the relative frequencies of the genetic types within the whole population in the initial generation. Denote by $p_{ij,k}$ the conditional probability that two individuals of type $i$, resp. $j$, produce an offspring of type $k$ given they interbreed and assume that the population is large enough for frequency fluctuations to be neglectable. Presuming a free population, i.e. absence of sexual differentiation and the statistical independence of genotypes for breeding, the distribution $\mathbf{x}' = (x'_1,...,x'_m) $ of the (expected) gene frequencies in the next generation is given by
\begin{equation}
\label{ng}
x'_k =\sum\limits_{i,j=1}^m p_{ij,k}x^0_ix^0_j, \  \  k \in \lsem m \rsem.
\end{equation}

The association $\mathbf{x}^0 {\mapsto} \mathbf{x}'$ defines a map {$V: S^{m-1} \to S^{m-1}$} called {\em evolutionary operator}. The population evolves by starting from an arbitrary frequency distribution $\mathbf{x}^0$, then passing to the state
$\mathbf{x}' = V (\mathbf{x}^0)$ in the next  {``}generation", then to the state
$\mathbf{x}'' = V (V (\mathbf{x}^0))$, and so on.
Thus {the evolution of gene frequencies} in this population can be considered as a dynamical system
$$
\mathbf{x}^0, \  \ \mathbf{x}' = V (\mathbf{x}^0), \  \ \mathbf{x}'' = V^2(\mathbf{x}^0), \  \ \mathbf{x}'''= V^3(\mathbf{x}^0), \  \ . . .
$$
Note that {$V$} as defined by (\ref{ng}) is a non-linear (quadratic) operator.  Higher dimensional dynamical
systems, as the one resulting from the observations above for $m \geq 3$, are important, but only relatively few dynamical phenomena are thoroughly comprehended (\!\!\cite{D}, \cite{E}, \cite{R}).

{One of the main objects of study for dynamical systems and QSOs is the asymptotic behaviour of their
trajectories, depending on the initial value. However, this has so far only been determined for certain particular subclasses of QSOs. One such subclass that arises naturally in the biological context is given by the additional restriction
\begin{equation}
\label{koefvolt}
p_{ij,k}=0, \ \ \mbox {if} \ \ k\notin \{i,j\}, \ \ i,j,k \in \lsem m \rsem.
\end{equation}
These QSOs describe a reproductory behaviour where the offspring is a genetic copy of one of its parents and are called {\em Volterra operators}. The asymptotic behaviour of trajectories of this kind of QSOs were analysed in \cite{RN1}, \cite{RN2} and \cite{RNEs} using the theory of Lyapunov functions and tournaments. In \cite{MAT} infinite dimensional Volterra operators were introduced and their dynamics studied. In \cite{GZ1}-\cite{GGJ},\cite{Zakh} the ergodicity problems of the Volterra operators were considered.
In \cite{GJM2} and \cite{RJ3} a Volterra operator of a bisexual population was examined.

However, in the non-Volterra case (i.e., where condition (\ref{koefvolt}) is violated), many questions remain open and there seems to be no general theory
available. See \cite{GMR} for a recent review of QSOs.\\

In all of the above-mentioned references the authors investigated deterministic trajectories of a QSO. However, it seems natural to consider a randomization of this procedure and explore the random dynamical system resulting from it. This can be done, e.g., by using a random iteration of operators of a given finite or countably infinite set of QSOs.

As a first step in this direction we investigate the trajectories of a sequence of independent and identically distributed Volterra QSOs in the present work. We prove that for any initial point from the simplex of probability distributions the random trajectory converges almost surely to one of the vertices of the simplex. This is far from being obvious since the set of Volterra QSOs considered may well contain operators that do not have this property and might, indeed, not converge at all. Furthermore, we show that the set of vertices of the simplex coincides with the minimal random point attractor of the corresponding random dynamical system.

Note that for the biological interpretation our results show that such a mechanism does not allow for coexistence but yields almost sure extinction of all but one species (Theorem \ref{thm:mr}). The corresponding results in the deterministic setting on the other hand cannot generally rule out coexistence in the long run (see, e.g., Proposition \ref{RN2} ). Indeed, some of the QSOs included in the set we consider for the random setting, e.g. those studied in \cite{Zakh}, model a very distinct deterministic behaviour. They describe a population where a species will come to the verge of extinction only to recover to the point where all other species are almost annihilated, after which the cycle repeats indefinitely, not yielding a stable situation.

The paper is organized as follows. In Section \ref{sec:preliminaries} we recall definitions and well
known results from the theory of Volterra QSOs and the definition of random QSOs.
In Section \ref{sec:mr} we define a special class of Volterra QSOs and show the almost sure convergence of the random iteration of these operators.
Finally, in Section \ref{sec:att}, we identify the minimal random point attractor of the resulting random dynamical system. In the appendix, we formulate and prove a martingale-type limit theorem which we need for the proof of the main result of Section \ref{sec:mr}.

\section{{Preliminaries} {and known results}}
\label{sec:preliminaries}
 A quadratic stochastic operator (QSO) on $\lsem m \rsem =\{1,\ldots,m\}$
 is a mapping $V$ of the simplex
\begin{equation}\label{simp}
S^{m-1}={\Big\{}\mathbf{x}=\big(x_1,...,x_m\big)\in {\mathbb{R}^m}: x_i\geq 0,\ \,{\forall  i \in \lsem m \rsem, }\, \sum_{i=1}^mx_i=1 {\Big\}}
\end{equation}
into itself, of the form $V(\mathbf{x})=\mathbf{x}' \in S^{m-1}$, where
\begin{equation}
\label{kso}
x'_k=\sum_{i,j \in \lsem m \rsem}p_{ij,k}x_i x_j, \ \ k \in \lsem m \rsem,
\end{equation}
{and the} $p_{ij,k}$  {satisfy}
\begin{equation}\label{koefkso}
p_{ij,k}=p_{ji,k}\geq 0, \quad  \ \ \sum_{k=1}^mp_{ij,k}=1, \ \ i,j,k \in \lsem m \rsem.
\end{equation}
The trajectory (orbit) $\{\mathbf{x}^{(n)}\}_{n \in \N_0}$ of $V$ for an initial {value} $\mathbf{x}^{(0)}\in S^{m-1}$  is defined by
\begin{equation}\label{trajectory}
\mathbf{x}^{(n+1)}=V(\mathbf{x}^{(n)})={V^{n+1}(\mathbf{x}^{(0)})}, \quad n=0,1,2,\dots
\end{equation}

The following notation will be used throughout this paper. We let ${\mathrm{int}}S^{m-1}$ denote the  interior and  $\partial S^{m-1} $ the boundary of $ S^{m-1}$, i.e.
\begin{equation*}
 {\mathrm{int}}S^{m-1}:=\{\mathbf{x}\in S^{m-1}:x_1x_2\cdots x_m>0\}, \text{ and }  \partial S^{m-1}:=S^{m-1} \setminus {\mathrm{int}} S^{m-1}.
\end{equation*}
Furthermore let $\mathbf{e}_i=(\delta_{1i},\delta_{2i}, \cdots, \delta_{mi})$ for $i=1,2,\cdots,m$ be used for the $i$th vertex of the simplex $S^{m-1},$ where $\delta_{ij}$ is the Kronecker symbol.  $\omega(\mathbf{x}^{(0)})$ denotes the  $\omega$-limit set of the trajectory (\ref{trajectory}).

A point $\mathbf{x}\in S^{m-1}$ is called a fixed point of $V$ if $V(\mathbf{x})=\mathbf{x}.$
Note that {our QSOs are continuous operators and that the simplex over a finite set is compact and convex, so that by the Brouwer Fixed-Point Theorem there is always at least one fixed point. Further, if a trajectory generated by the QSO $V$ converges to $\mathbf{x}$ then, by continuity, $\mathbf{x}$ is a fixed point.

\vskip 0.5 truecm

{ \it Volterra Quadratic Stochastic Operators}\\

Let $V$ be a quadratic stochastic operator on the simplex $S^{m-1}.$

 \begin{defn}\label{dvqso}
  The quadratic stochastic operator $V$  is called \emph{Volterra operator}, if $p_{ij,k}=0 $ for any $k\notin \{i,j\}$, $ \ \ i,j,k \in \lsem m \rsem$.
 \end{defn}
 Evidently for any Volterra QSO
 \begin{equation} \label{volkoef}
 p_{ii,i}=1, \mbox{ and }     p_{ik,k}+p_{ki,i}= p_{ik,k}+p_{ik,i}=1 \ \ \mbox {for all } i,k \in \lsem m \rsem,  i\neq k.
 \end{equation}
 A Volterra QSO $V$ defined on $S^{m-1}$ therefore has the following form
 \begin{equation}\label{volgf}
(V \mathbf{x})_k=x_k^2+2\sum_{i\in \lsem m \rsem,i\neq k} p_{ik,k}x_ix_k, \qquad k \in \lsem m \rsem.
\end{equation}

\begin{pro} \cite{RN1} \label{RN1} A QSO $V$ is a Volterra operator if and only if
 \begin{equation}\label{gvol}
(V \mathbf{x})_k=x_k(1+\sum_{i=1}^ma_{ki}x_i)
\end{equation}
where $A=(a_{ij})_1^m$ is a skew-symmetric matrix with $a_{ki}=2p_{ik,k}-1$ for $i\neq k,$ $a_{ii}=0$   and $|a_{ij}|\leq 1.$ Here $i,j,k\in \lsem m \rsem.$
\end{pro}

The space of skew-symmetric matrices generating Volterra operators, is parameterized by the cube $[-1,1]^{m(m-1)/2}$.
The extremal points of the cube are its vertices. The quadratic stochastic operator $V$ is called an {\it extremal Volterra operator},
if the corresponding skew-symmetric matrix is a vertex of the cube, i.e. $a_{ij}=-1  \mbox{ or } 1 $ for any  $i\neq j.$

It is evident that the total number of the extremal Volterra QSO is equal to $2^{\frac{m(m-1)}{2}}.$

\begin{pro} \cite{RN1}\label{RN2} Let $V$ be a Volterra QSO. Then
\begin{itemize}
 \item $V$ is a homeomorphism on $S^{m-1}$;
 \item If $\mathbf{x}$ is not a fixed point of $V$, then $\omega(\mathbf{x})\subset \partial S^{m-1}$.
\end{itemize}
\end{pro}

\begin{pro}\label{volotsen}
For any Volterra operator  $V$ and any $k \in \lsem m \rsem$, we have
 \begin{equation}\label{volest}
(V \mathbf{x})_k\leq 2x_k.
\end{equation}
\end{pro}
\begin{proof}
The proof immediately follows from Proposition \ref{RN1}.
\end{proof}

\vskip 0.5 truecm

{ \it Random Quadratic Stochastic Operators}\\

  In this subsection we recall the definition of a {\it random quadratic stochastic operator} following \cite{NN}.
Let $\Upsilon$  be  the set of all quadratic stochastic operators defined on $S^{m-1}$. Since every QSO is represented by a cubic
matrix $(p_{ij,k})_{i,j,k\in \lsem m \rsem}$ the set $\Upsilon$ is compactly embedded in $\R^{m^3}$. Let $\mathcal{H}$
be the Borel $\sigma$-algebra induced on the set $\Upsilon$.

\begin{defn}\cite{NN}\label{NN}
Consider a probability space $(\Omega,\mathfrak{F},\mathbb{P})$. Any measurable map $T:\Omega\rightarrow \Upsilon$
(i.e. such that $T^{-1}(\mathcal{H})\subset \mathfrak{F} $) is called a {\em random quadratic stochastic operator} (RQSO).
\end{defn}

In \cite{NN} a class of dyadic random quadratic stochastic operators in random environment was investigated.

\section{Main Result}
\label{sec:mr}
Let $\mathcal{V}$ be a countable set of Volterra QSOs on $S^{m-1}$ such that for each $k\in \lsem m \rsem$ there exists a $V \in \mathcal V$ such that
\begin{equation}
\label{eq:squares}
(V \mathbf{x})_k=x^2_k.
\end{equation}
Assume $\mathcal V$ to be indexed by $\mathbb N$ such that \eqref{eq:squares} holds for the corresponding $V_1, \ldots, V_m \in \mathcal V$.
Note that such Volterra QSOs exist -- even extremal ones: in fact \eqref{eq:squares}  holds for $V$ if and only if the associated skew-symmetric matrix
$A$ in Proposition \ref{RN1} satisfies $a_{ki}=-1$ for all $i \neq k$. The skew-symmetry of $A$ also shows that no Volterra QSO $V$ can satisfy
\eqref{eq:squares} for two different values of $k$.

Let $\nu_i, i=1,2,...$ be a probability distribution on $\mathcal V$ such that $\nu_i>0$ for all $i\in \lsem m \rsem$.

\begin{thm}\label{thm:mr}
Consider a sequence  $T_1,T_2, ....$ of independent RQSO in $\mathcal{V}$ such that $\mathbb{P}(T_i=V_j)=\nu_j$ for each $j=1,2,...$ and $i \in \N$.
Then, for any $\mathbf{x}\in S^{m-1}$, we have that
$$
\mathbb{P}\big(\lim\limits_{n\rightarrow \infty}(T_n \circ \ldots \circ T_1)(\mathbf{x}) \in \{\mathbf{e}_1,...,\mathbf{e}_m\} \big)=1.
$$

\end{thm}

For $\varepsilon>0$ we denote by $U^\varepsilon_i=\{\mathbf{x}\in S^{m-1}: x_j<\varepsilon, \ \ j\in \lsem m \rsem\setminus \{i\}\}$ the
$\varepsilon$-neighborhood of the vertex $\mathbf{e}_i, \ i\in \lsem m \rsem$ and $U_\varepsilon=\bigcup\limits_{i\in \lsem m \rsem} U^\varepsilon_i$. Further, we define $\Lambda:=\{\mathbf{e}_1,...,\mathbf{e}_m\}$
as the set of vertices of $S^{m-1}$.
The following proposition shows that for given $\varepsilon >0$ one can find some $N$ such that after $N$ iterations the probability of ending up in $U_{\varepsilon}$ is
bounded away from 0 uniformly with respect to the initial condition. Proposition \ref{pro2} then shows that in the case of this event there is a certain chance that the trajectory then converges to the corresponding vertex. To show the main result, we then argue that if this fails (i.e. either the trajectory is not in $U_{\varepsilon}$
after $N$ iterations or it is but the trajectory then leaves the neighborhood rather than converging to the corresponding vertex) we simply try once more. Since the chance
of being successful is bounded away from zero uniformly for all starting points, the result then follows.

\begin{pro}\label{pro1}
Under the assumptions from the Theorem \ref{thm:mr} we know that for each $\varepsilon>0$ there are $N \in\mathbb{N}$ and $q>0$ such that for every point $\mathbf{x}\in S^{m-1}$
$$
\mathbb{P}\Big(T_{N} \circ T_{N-1}\circ ...\circ T_1(\mathbf{x})\in U_\varepsilon\Big)\geq q.
$$
\end{pro}
\begin{proof}
Let $\varepsilon >0$ and choose $r\in \N$ so large that $-2^r +(m-2)r < \log(\varepsilon)/\log(2)$.

Now fix some starting point $\mathbf{x} \in S^{m-1}$ and define $j_1 \in \lsem m \rsem$ as the index of the vertex corresponding to the maximal distance of $\mathbf{x}$ to $\Lambda$, i.e.
\begin{align*}
 \Vert \mathbf{x}-\mathbf{e}_{j_1}\Vert = \max_{j \in \lsem m \rsem}\Vert \mathbf{x} - \mathbf{e}_j\Vert.
\end{align*}
We first want to find a deterministic sequence  $\bar V_1, \ldots, \bar V_{m-1} \in \mathcal V$ such that $\bar V_{m-1}^r\circ\ldots\circ\bar V_1^r (\mathbf{x}) \in U_{\varepsilon}^{j_0}$ in order to then prove that the probability of this realization is bounded away from 0.
Begin by setting $\bar V_1 :=V_{j_1}$ and define $j_2 \in \lsem m \rsem$ as the index corresponding to
\begin{align*}
 \Vert \bar V_1^r(\mathbf{x}) - \mathbf{e}_{j_2}\Vert = \max_{j \in \lsem m \rsem\setminus \{j_1\}} \Vert \bar V_1^r(\mathbf{x}) - \mathbf{e}_j\Vert
\end{align*}
and $\bar V_2:=V_{j_2}$, then iterate this construction. Define $J_k:=\{j_1,\ldots,j_k\}$ and let $j_{k+1}\in \lsem m \rsem$ be the index corresponding to
\begin{align*}
 \Vert \bar V_k^r\circ\ldots\circ\bar V_1^r(\mathbf{x}) - \mathbf{e}_{j_{k+1}}\Vert = \max_{j \in \lsem m \rsem\setminus J_k} \Vert \bar V_k^r\circ\ldots\circ\bar V_1^r(\mathbf{x}) - \mathbf{e}_j\Vert
\end{align*}
and set $\bar V_{k+1} :=V_{j_{k+1}}$ for $k = 2, \ldots, m-2$. Observe that we have chosen the indices such that $(\bar V_{k-1}^r\circ \ldots\circ\bar V_1^r (\mathbf{x}))_{j_k} \leq 1/2$ for all $k = 2, \ldots, m-1$.

Since \eqref{eq:squares} holds for our $\bar V_1, \ldots, \bar V_{m-1} \in \mathcal V$ we obtain the following estimates for every $k = 1, \ldots, m-1$:
\begin{align*}
 (\bar V_{m-1}^r\circ\ldots\circ\bar V_1^r (\mathbf{x}))_{j_k}	& = (\bar V_{m-1}^r\circ\ldots\circ\bar V_{k+1}^r(\bar V_k^r(\bar V_{k-1}^r\circ \ldots\circ\bar V_1^r (\mathbf{x}))))_{j_k}\\
							& \leq 2^{r(m-1-k)} (\bar V_{k-1}^r\circ \ldots\circ\bar V_1^r (\mathbf{x}))_{j_k}^{(2^r)}\\
							& \leq 2^{r(m-1-k)}\left(\frac{1}{2}\right)^{(2^r)} = 2^{r(m-1-k)-2^r} \leq  2^{r(m-2)-2^r} < \varepsilon
\end{align*}
where we used Proposition \ref{volotsen} in the first inequality. This implies
\begin{align*}
 \bar V_{m-1}^r\circ\ldots\circ\bar V_1^r (\mathbf{x}) \in U_{\varepsilon}^{j_0}.
\end{align*}
Observe that the probability of choosing these operators can be estimated due to the independence assumption by
\begin{align*}
 \mathbb P ( \forall k=1,\ldots,m-1, \forall (k-1)r<s\leq kr:\; T_s = \bar V_k) = \nu_{j_{m-1}}^r\ldots \nu^r_{j_1} \geq \nu_1^r\ldots \nu_m^r
\end{align*}
where the last estimate does not depend on the starting point $\mathbf{x}\in S^{m-1}$ anymore. Therefore $N:=r(m-1)$ and $q:=\nu_1^r\ldots \nu_m^r$ fulfill the claim.
\end{proof}

In order to analyze the convergence consider a sequence $(T_n)_{n \in \mathbb N}$ of random QSOs as in Theorem \ref{thm:mr} and let $X$ denote a random variable taking values in $\text{int}S^{m-1}$ that is independent of the sequence and such that $\mathbb E[\vert\log(X)\vert]<\infty$. Define a filtration $(\mathcal F_n)_{n \in \N_0}$ by $\mathcal F_n := \sigma(X,T_1, \ldots, T_n)$ for $n \in \N_0$. We introduce the abbreviation $\hat T_n:= T_n\circ\ldots\circ T_1$ and use this to define
\begin{align}
\label{eq:Z}
 Z^i_n:=\log((\hat T_nX)_i).
\end{align}
Note that, by Proposition \ref{RN1} $\hat T_nX \in \text{int} S^{m-1}$ for all $n \in \N$ and thus \eqref{eq:Z} is well-defined.

We would like the increments of this process to be (at least) integrable, but since this is not necessarily the case we define a new process $(Y^i_n)_{n \in \N_0}$ in the following way: Choose $d > \max\{\log(m),\max_{i \in \lsem m \rsem}\left\{1/\nu_i\right\}\log(2)\}$ and set
\begin{align}
 Y^i_0	& := Z^i_0 = \log(X^i), \\
 Y^i_{n+1}-Y^i_n & :=\begin{cases}
                Z^i_{n+1} -Z^i_n,	&\text{if } Z^i_{n+1} -Z^i_n \geq -d,\\
                -d,	& \text{otherwise.}
               \end{cases}
\end{align}
Then we know that for all $\omega \in \Omega$: $Z^i_n(\omega) \leq Y^i_n(\omega)$.

\begin{pro}
\label{pro2}
 For $D:=\min_{i\in \lsem m \rsem}\{\nu_id-\log(2)\}>0$ we have for every $j \in \lsem m \rsem$
 \begin{align*}
 \mathbb P\big(\forall i\in\lsem m \rsem\setminus\{j\}:\liminf_{n \rightarrow \infty} -\frac{1}{n}Y^i_n \geq D \,\big|\, \forall i\in\lsem m \rsem\setminus\{j\}:\forall n \in \mathbb N:\, Y^i_n \leq -d\big) = 1.
\end{align*}
Moreover for every $\theta >0$ and every $b \in \mathbb R$ there exists an $s>0$, such that
\begin{align*}
 \mathbb P( \exists j\in\lsem m \rsem\forall i\in\lsem m \rsem\setminus\{j\}\, \forall n \in \mathbb N:\; Y^i_n < b\mid \mathcal F_0) \geq 1-\theta \;\text{ on } \{X \in \bar U_s\} \ \mathbb P\text{-a.s.}
\end{align*}
where $ \bar U_s :=\{\mathbf{x}\in S^{m-1}: \exists j \in \lsem m \rsem \forall i \neq j: x_j\leq\varepsilon\}$
\end{pro}

\begin{proof}
 Note that the increments of $(Y^i_n)_{n \in \N_0}$ are integrable.
Thus we can calculate
\begin{align*}
 \mathbb E[Y^i_{n+1}-Y^i_n\mid \mathcal F_n] & = \mathbb E[\log\left(\frac{T_{n+1}(\hat T_n(X))_i}{\hat T_n(X)_i}\right)\lor (-d)\mid \mathcal F_n] \\
					& = \nu_i\left(\log\left(\hat T_n(X)_i\right)\lor (-d)\right) + \sum_{j\neq i}\nu_j \log\underbrace{\left(\frac{V_j(\hat T_n(X))_i}{\hat T_n(X)_i}\right)}_{\leq 2 \text{ by Prop. \ref{volotsen}}} \\
					& \leq \nu_i\left(Z^i_n \lor (-d)\right) +\log(2)\\
					  & \leq - \nu_id+\log(2) \leq - D \text{ on } \{Z^i_n \leq -d\} \text{ and thus also on } \{Y^i_n \leq -d\}
\end{align*}
and
\begin{align*}
 \mathbb E[(Y^i_{n+1}-\mathbb E[Y^i_{n+1}\mid\mathcal F_n])^2\mid \mathcal F_n]	& = \mathbb E[(Y^i_{n+1} - Y^i_n-\mathbb E[Y^i_{n+1}-Y^i_n\mid\mathcal F_n])^2\mid \mathcal F_n] \\
										 & = \mathbb E[(Y^i_{n+1} - Y^i_n)^2\mid \mathcal F_n] -\mathbb E[(\mathbb E[Y^i_{n+1}-Y^i_n\mid\mathcal F_n])^2\mid \mathcal F_n]\\
										 & \leq \mathbb E[(\underbrace{(Y^i_{n+1} - Y^i_n)^+}_{\leq \log(2)})^2\mid \mathcal F_n] + \mathbb E[(\underbrace{(Y^i_{n+1} - Y^i_n)^-}_{\leq d})^2\mid \mathcal F_n]\\
									      & \leq (\log(2))^2 +d^2 \qquad \mathbb P\text{-a.s.}
\end{align*}
This allows us to apply Proposition \ref{prop:appendix} yielding
\begin{align}
\label{eq:limY_i}
 \mathbb P(\liminf_{n \rightarrow \infty} -\frac{1}{n}Y^i_n \geq D \mid \forall n \in \mathbb N:\; Y^i_n \leq -d) = 1
\end{align}
and that for every $\theta >0$ and every $b \in \mathbb R$ there exists an $r_i \in \R$, such that
\begin{align}
\label{eq:Y_i}
 \mathbb P(\forall n \in \mathbb N: Y^i_n < b\mid \mathcal F_0) \geq 1-\frac{1}{m-1}\theta \; \text{ on } \{Y^i_0 \leq r_i\}=\{\log(X^i) \leq r_i\}.
\end{align}
From \eqref{eq:limY_i} we obtain for every $j \in \lsem m \rsem$
\begin{align*}
 \mathbb P(\forall i\in\lsem m \rsem\setminus\{j\}:\,\liminf_{n \rightarrow \infty} -\frac{1}{n}Y^i_n \geq D \mid \forall i\in\lsem m \rsem\setminus\{j\}\,\forall n \in \mathbb N:\; Y^i_n \leq -d) = 1.
\end{align*}
With $s:=\min_{i=1, \ldots, n}\{\exp(r_i)\}$ for any $j \in \lsem m \rsem$ \eqref{eq:Y_i} implies
\begin{align*}
 \mathbb P(\forall i\in\lsem m \rsem\setminus\{j\}\,\forall n \in \mathbb N: Y^i_n < b\mid \mathcal F_0) \geq 1-\theta \text{ on } \{X \in \bar U^j_s\}=\bigcap_{i\in\lsem m \rsem\setminus\{j\}} \{X^i \leq s\}
\end{align*}
and thus
\begin{align*}
 \mathbb P(\exists j \in \lsem m \rsem\,\forall i\in\lsem m \rsem\setminus\{j\}:\,\forall n \in \mathbb N: Y^i_n < b\mid \mathcal F_0) \geq 1-\theta \text{ on } \{X \in \bar U_s\}.
\end{align*}
\end{proof}

\begin{proof}[Proof of Theorem \ref{thm:mr}]

Recall the definitions of $D$ and $d$ from above. Note that by Proposition \ref{RN1} for any $k \in \lsem m \rsem$ and every Volterra operator $V$ $x_k \neq 0$ if, and only if $(V \mathbf{x})_k \neq 0$. Thus, by disregarding the zero-entries,  starting on $\partial S^{m-1}$ can be interpreted as starting and considering the same problem on the interior of a lower-dimensional simplex. Therefore, w.l.o.g. we can assume $\mathbf{x} \in \text{int} S^{m-1}$. Let $\theta \in (0,1)$ be arbitrary and setting $b:=-d$ choose $s$ as in Proposition \ref{pro2}. For $\varepsilon:=\min\{s, \frac{1}{m}\}$ let $N$ and $q$ be as in Proposition \ref{pro1}.

 We begin by defining the objects we will need for the proof. Define the stopping time
 \begin{align*}
\tau_1	&:= \inf\{nN\mid \exists j \in \lsem m \rsem\,\forall i\in\lsem m \rsem\setminus\{j\}: Z^i_{nN} < \log(\varepsilon)\} = \inf\{nN\mid \hat T_{nN}(x) \in U_{\varepsilon}\}.
 \end{align*}
Note that Proposition \ref{pro1} shows that $\tau_1$ is almost surely finite.
Set $J_1:=\min\{j \in \lsem m \rsem\mid \hat T_{\tau_1}(x) \in U_{\varepsilon}^j\}$.
Now for every index $i \neq J_1$ we start the cut-off version $(Y^{\tau_1}_n)_{n \in \N_0}$ of our process given by
\begin{align*}
Y^{\tau_1,i}_0	& := \log(\varepsilon) \geq Z^i_{\tau_1}, \\
 Y^{\tau_1,i}_{n+1}-Y^{\tau_1,i}_n & :=\begin{cases}
                Z^i_{\tau_1+n+1} -Z^i_{\tau_1+n},	&\text{if } Z^i_{\tau_1+n+1} -Z^i_{\tau_1+n} \geq -d\\
                -d,	& \text{otherwise}
               \end{cases}	
\end{align*}
for all $n \in \N_0$ and use this to define the stopping time
\begin{align*}
 \sigma_1 	:= \inf\{n > \tau_1 \mid \exists i \neq J_1: Y^{\tau_1,i}_n \geq -d\}.
\end{align*}
$J_1$ and $\sigma_1$ are well-defined since $\tau_1 < \infty$ $\mathbb P$-a.s.
Recursively then define
\begin{align*}
 \tau_{k+1} & := \inf\{nN>\sigma_k\mid \exists j \in \lsem m \rsem:\forall i\in\lsem m \rsem\setminus\{j\}: Z^i_{nN} < \log(\varepsilon)\} \\
	    & = \inf\{nN>\sigma_k\mid \hat T_{nN}(x) \in U_{\varepsilon}\} \\
 J_{k+1}	& := \min\{j \in \lsem m \rsem \mid \hat T_{\tau_{k+1}}(x)_{\tau_{k+1}} \in U_{\varepsilon}^j\} \\
	    Y^{\tau_{k+1},i}_0	& := \log(s) \geq Z^i_{\tau_{k+1}} \\
 Y^{\tau_{k+1},i}_{n+1}-Y^{\tau_{k+1},i}_n  &:=\begin{cases}
                Z^i_{\tau_{k+1}+n+1} -Z^i_{\tau_{k+1}+n},	&\text{if } Z^i_{\tau_{k+1}+n+1} -Z^i_{\tau_{k+1}+n} \geq -d\\
                 -d,	& \text{otherwise,}
               \end{cases}	\\
  \sigma_{k+1} 	&:=\inf\{n > \tau_{k+1} \mid \exists i \neq J_{k+1}: Y^{\tau_{k+1},i}_n \geq -d\}
\end{align*}
for $i \neq J_{k+1}$, $n \in \N_0$.

Note that, on $\{\sigma_k=\infty\}$ we have the existence of a $j\in\lsem m \rsem$ such that for all other $i\in\lsem m \rsem\setminus\{j\}:\,Y^{\tau_k,i}_n < - d$ holds, which by Proposition \ref{pro2} and its definition implies that $\lim_{n \rightarrow \infty} Z^i_n =-\infty$. This is, however, equivalent to $\lim_{n \rightarrow \infty} \hat T(x) \in \Lambda$, our desired result.

Of course, since some of the above are only well-defined, when the corresponding stopping times are finite, we begin by considering the probabilities of these events.
Again, by Proposition \ref{pro1} we know that $\mathbb P(\tau_{k+1} < \infty \mid \mathcal F_{\sigma_k}) = 1$ on $\{\sigma_k < \infty\}$.
Furthermore, since $\{\tau_k<\infty\} \subset\{\hat T_{\tau_k}(x) \in U_{\varepsilon}\} $ we know that
\begin{align*}
 \mathbb P(\sigma_k = \infty \mid \mathcal F_{\tau_k})	& = \mathbb P(\exists j\in\lsem m \rsem:\forall i\in\lsem m \rsem\setminus\{j\}: \forall n\in \N_0: Y^{\tau_k, i}_n < -d\mid \mathcal F_{\tau_k}) \\
							& \geq 1-\theta
\end{align*}
on $\{\tau_k < \infty\}$ $\mathbb P$-a.s. by Proposition \ref{pro2} and therefore $\mathbb P(\sigma_k < \infty \mid \mathcal F_{\tau_k}) \leq \theta$ on $\{\tau_k < \infty\}$.
Combining the results above we see that for every $k\in \N$ we have $ \mathbb P(\sigma_k <\infty \mid \mathcal F_{\sigma_{k-1}}) \leq \theta \text{ on } \{\sigma_{k-1}<\infty\}$
which we can use to conclude
\begin{align*}
 \mathbb P(\sigma_k<\infty)	& = \mathbb P(\sigma_k<\infty, \ldots, \sigma_1<\infty) \\
			& = \mathbb E(\underbrace{\mathbb P(\sigma_k<\infty\mid\mathcal F_{\sigma_{k-1}})}_{\leq \theta}\mathbbm 1_{\{\sigma_{k-1}<\infty\}}  \cdots\mathbbm 1_{\{\sigma_1<\infty\}} )\\
			& \leq \theta \mathbb P(\sigma_{k-1}<\infty, \ldots, \sigma_1<\infty) \leq \theta^k
\end{align*}
iterating the argument used in the last step. Therefore $ \sum_{k\in \N} \mathbb P(\sigma_k<\infty) < \infty$
which by the first Borel-Cantelli Lemma implies that $\mathbb P(\exists k\in \N:\, \sigma_k=\infty)=1$.
Since we chose the $\{\forall i\in\lsem m \rsem\setminus\{j\}:\,\forall n \in \mathbb N:\; Y^i_n < -d\}_{j \in \lsem m\rsem}$ to be disjoint by Proposition \ref{pro2} we know that
 \begin{align*}
 \mathbb P(\exists j \in \lsem m \rsem\,\forall i\in\lsem m \rsem\setminus\{j\}:\,\lim_{n \rightarrow \infty} Y^{\tau_k,i}_n = -\infty \mid \sigma_k = \infty) = 1
 \end{align*}
and can conclude
\begin{align*}
 1	& = \mathbb P(\exists k \in \N\,\exists j \in \lsem m \rsem\,\forall i\in\lsem m \rsem\setminus\{j\}:\, \lim_{n\rightarrow\infty}Y^{\tau_k, i}_n = -\infty) \\
	& \leq \mathbb P(\exists j \in \lsem m \rsem\,\forall i\in\lsem m \rsem\setminus\{j\}:\, \lim_{n \rightarrow \infty}Z^i_n = -\infty)\\
	& \leq \mathbb P(\lim_{n \rightarrow \infty} \hat T_n(x) \in \Lambda)
\end{align*}
which completes the proof of  Theorem \ref{thm:mr}.
\end{proof}

\section{Random Attractors}\label{sec:att}
In this section, we recall the concept of a random attractor of a random dynamical system (RDS) and show that the RDS generated by
the sequence of random operators in Theorem \ref{thm:mr} has the set $\Lambda=\{\mathbf{e}_1,...,\mathbf{e}_m\}$ as a minimal random point attractor.
There exist a number of different concepts of random attractors some of
which we will introduce below. We restrict our attention to the discrete time setting.

Let $(E,d)$ be a separable, complete metric space and denote its Borel-$\sigma$-field by $\E$. The following definition can be found in \cite{Ar98}.

\begin{defn}

\begin{itemize}
\item[a)]
  $\bigl(\Omega,\F,\PP,(\vartheta_n)_{n\in\mathbb Z}\bigr)$
  is called a \emph{metric dynamical system (MDS)}, if
  $(\Omega,\F,\PP)$ is a probability space, and the family of
  maps
  $\bigl\{\vartheta_n:\Omega\to\Omega,n\in\mathbb Z\bigr\}$
  satisfies
  \begin{enumerate}
  \item[(i)] the mapping $\omega\mapsto\vartheta_n(\omega)$
    is measurable for each $n \in \N_0$,
  \item[(ii)] $\vartheta_{m+n}=\vartheta_m\circ\vartheta_n$ for
    every $m,n\in\mathbb Z$, and $\vartheta_0=\mbox{Id}_\Omega$,
    and
  \item[(iii)] for each $n\in\N_0$, $\vartheta_n$ preserves the
    measure~$\PP$.
\end{enumerate}
\item[b)] A \emph{random dynamical system (RDS)} on the measurable space
  $(E,\E)$ over the MDS
  $\bigl(\Omega,\F,P,(\vartheta_n)\bigr)$ with time~$\N_0$
  is a mapping
  \begin{displaymath}
    \varphi:\N_0\times E\times\Omega \to E, \qquad
    (n,\mathbf{x},\omega)\mapsto \varphi(n,\mathbf{x},\omega)
  \end{displaymath}
  with the following properties:
  \begin{enumerate}
  \item[(i)] For each $n \in \N_0$, $\varphi(n,.,.)$ is
    $\left(
      \E\otimes\F,\E\right)$-measurable.
  \item[(ii)] For all $m,n\in \N_0)$
    \begin{displaymath}
      \varphi(m+n,\omega)
      =\varphi(m,\vartheta_n\omega)\circ\varphi(n,\omega)\qquad
        \mbox{for all}\ \omega\in\Omega,
   \end{displaymath}
   and $\varphi(0,\omega)=\mbox{Id}_E$ for all $\omega\in\Omega$.
 \end{enumerate}
The RDS~$\varphi$ is called \emph{continuous} if, in addition,
 \begin{enumerate}[resume]
 \item[(iii)]the mapping $\mathbf{x}\mapsto\varphi(n,\mathbf{x},\omega)$ is continuous
   for all $(n,\omega)\in \N_0\times\Omega$.
 \end{enumerate}
\end{itemize}
\end{defn}

The following definition of a global attractor is (essentially) due to Crauel and
Flandoli~\cite{CF94} while point attractors were introduced in \cite{Crauel01}.

\begin{defn}  \label{f_ibf_watr}
Let~$\varphi$ be an RDS on~$E$ over the MDS
$\bigl(\Omega,\F,\PP,(\vartheta_n)_{n\in\N_0}\bigr)$.
Let $\mathcal B\subset2^E$ be an arbitrary subset of the power set of~$E$.
A family of sets $A(\omega) \in 2^E, \omega \in \Omega$ is called a
\emph{$\mathcal B$-attractor} for~$\varphi$ if
\begin{itemize}
\item $A$ is a compact random set (i.e.\ $A(\omega)$ is nonempty and compact for every $\omega \in \Omega$ and
$\omega\mapsto d\bigl(\mathbf{x},A(\omega)\bigr)$ is measurable for every
$\mathbf{x}\in E$).
\item $A$ is \emph{strictly $\varphi$-invariant}, that is, there exists a set $\Omega_0$ of full measure, such
  that
  $\varphi(n,\omega)(A(\omega))=A(\vartheta_n\omega)$ for all $n \in \N_0$,
  $\omega\in\Omega_0$.
\item $\displaystyle\lim_{n\to\infty}\sup_{\mathbf{x}\in B}
  d\bigl(\varphi(n,\vartheta_{-n}\omega)(\mathbf{x}),A(\omega)\bigr)=0$
  almost surely for every $B\in\mathcal B$.
\end{itemize}
In particular, a $\mathcal B$-attractor is called
\begin{itemize}
\item \emph{global attractor} in case that~$\mathcal B$ is the
  set of all compact subsets of~$E$,
\item \emph{point attractor} in case that~$\mathcal B$ is the
  set of all singletons $\{ \{\mathbf{x}\},\, \mathbf{x} \in E\}$ (or -- equivalently -- the set of all finite subsets of $E$).
\end{itemize}
\end{defn}

A random attractor as introduced in the previous definition is often called {\em strong attractor} or
{\em pullback attractor} as opposed to a {\em weak attractor} for which the almost sure convergence is relaxed to
convergence in probability. One can argue that weak attractors occur more naturally than strong ones (see e.g. \cite{ChuScheu04}) (but proving the existence of a strong attractor is of course a stronger statement).
Sometimes the word {\em compact} is replaced by {\em bounded} in the definition of
a global attractor.
While a global attractor, if it exists, is always unique (up to sets of measure zero, see \cite{Crauel99}) this
is not true for a point attractor (Theorem \ref{attr} below provides an example). We call a point attractor $A(\omega)$ {\em minimal}
if for every other point attractor $\tilde A(\omega)$, we have $A(\omega) \subseteq \tilde A(\omega)$ for almost all $\omega \in \Omega$.
Under mild assumptions, existence of a point attractor implies the existence of a minimal point attractor
(see \cite{Crauel01}, Remark 3.5 (iii)).
Clearly, each global attractor is also a point attractor but the converse is not necessarily true (again, Theorem \ref{attr} below provides an example). Note that a comparison between different concepts of a random attractor has been performed in a special case in \cite{Scheu02} and
criteria for strong and weak random attractors have been established in \cite{CDS09}.\\

We are now ready to apply the concepts to the system introduced in the previous sections. We assume that
all assumptions in Theorem \ref{thm:mr} hold.
As the basic probability space $(\Omega,\F,\PP)$ we take $(\Omega,\F,\PP):=(\V,\nu)^{\Z}$ (where $\V$ is equipped with
the $\sigma$-field of all subsets of $\V$). Further, we define $\big(\vartheta_n(\omega)\big)_m=\omega_{m+n}$, $m,n \in \Z$. Then,
$\bigr(\Omega,\F,\PP,(\vartheta_n)_{n\in\N_0}\bigl)$ is a metric dynamical system and
$\varphi(n,\omega,\mathbf{x}):=\omega_n\circ...\circ \omega_1(\mathbf{x}),\;n \in \N_0,\;\mathbf{x} \in S^{m-1}$
defines an $S^{m-1}$-valued continuous RDS. Since  $S^{m-1}$ is compact and all $V \in \V$ are homeomorphisms, it follows that
$A(\omega):=S^{m-1}$ is the random attractor of $\varphi$. It turns out that $A$ is however not the minimal point attractor.

\begin{thm}\label{attr}
In the set-up above, the set  $\Lambda=\{\mathbf{e}_1,...,\mathbf{e}_m\}$ is  the minimal point attractor of the RDS $\varphi$.
\end{thm}

\begin{proof}
The measurability and invariance properties of a point attractor clearly hold for $\Lambda$. Further, each point attractor has to contain
$\Lambda$ since each point in $\Lambda$ is invariant under every $V \in \V$.  Therefore, it only remains  to show that for each $\mathbf{x} \in S^{m-1}$, we have
$$
\lim_{n\to\infty}
  d\bigl(\varphi(n,\vartheta_{-n}\omega)(\mathbf{x}),\Lambda\bigr)=0,\qquad \PP\mbox{-a.s.}
$$
If we replace ``$\PP$-a.s.'' by ``in probability'', then the result follows immediately from Theorem \ref{thm:mr}. In order to infer
almost sure convergence from convergence in probability, it suffices to show that convergence in probability happens sufficiently quickly.
In fact, thanks to the first Borel-Cantelli Lemma, it suffices to prove that for each $\varepsilon>0$, we have
$$
\sum_{n=1}^{\infty} \PP\big( T_n \circ ... \circ T_1(\mathbf{x}) \in U_{\varepsilon}\big) <\infty.
$$
Observe that Propositions \ref{pro1} and \ref{pro2} together show that the summands converge to zero exponentially quickly and therefore
the assertion follows.
\end{proof}

\section{Appendix}

\begin{pro}
\label{prop:appendix}
 Consider a real-valued process $(Y_n)_{n \in \mathbb N_0}$ that is in $\mathcal L^1(\mathbb P)$ and adapted to a filtration $(\mathcal F_n)_{n \in \mathbb N_0}$ such that for some $a \in \mathbb R$ and $A,B >0$ we have that for all $n \in \mathbb N_0: $
 \begin{enumerate}
  \item $\mathbb E[Y_{n+1}\mid\mathcal F_n] \geq Y_n + A$ and
  \item $\mathbb E[(Y_{n+1} -\mathbb E[Y_{n+1} \mid \mathcal F_n])^2\mid \mathcal F_n] \leq B$
 \end{enumerate}
  on $\{Y_n \geq a\}$ $\mathbb P$-a.s.
Then
\begin{align}
\label{eq:appendix1}
 \mathbb P(\liminf_{n \rightarrow \infty} \frac{1}{n} Y_n \geq A \mid \forall n \in \mathbb N:\; Y_n \geq a) = 1.
\end{align}
Moreover for every $\theta >0$ and every $b \in \mathbb R$ there exists an $S \in \R$, such that
\begin{align}
\label{eq:appendix2}
 \mathbb P(\forall n \in \mathbb N: Y_n > b\mid \mathcal F_0) \geq 1-\theta
\end{align}
$\mathbb P$-a.s. on $\{Y_0 \geq S\}$.
\end{pro}

\begin{proof}
The proof of \eqref{eq:appendix2} follows an idea of Rajchman used to prove a strong law of large numbers, see \cite[Theorem 2.14]{Krengel}. A similar result with stronger assumptions is given in \cite[Lemma 2.6]{ScheuStein}.

We begin with the proof of the first statement and define $\tau:=\{n \in \mathbb N\mid Y_n < a\}$ as the first time our process jumps below the level $a$.

We will want to apply Theorem 2.19 from \cite{HH} to the sequence $((Y_{n+1} - \mathbb E[Y_{n+1}\mid\mathcal F_n])\mathbbm{1}_{\{\tau > n\}})_{n \in \mathbb N_0}$.

Therefore let $\Xi$ be a random variable such that $\mathbb P(\Xi \leq 1)=0$ and $\mathbb P(\Xi >x)=\frac{1}{x^2}$ for $x>1$. Then $\mathbb E[\Xi\log^+\Xi]<\infty$ and since
\begin{align*}
\label{eq:appHH}
 \mathbb P (\vert Y_{n+1} - \mathbb E[Y_{n+1}\mid\mathcal F_n]\vert\mathbbm{1}_{\{\tau>n\}}>x)	 & \leq  \left(\mathbb E[(Y_{n+1} - \mathbb E[Y_{n+1}\mid\mathcal F_n])^2\mathbbm{1}_{\{\tau>n\}}]\frac{1}{x^2}\right)\land1\\
												  &\leq \left(B \frac{1}{x^2}\right)\land1  = (B\lor1)\mathbb P(\Xi>x)
\end{align*}
for all $x>0$ and $n \in \mathbb N_0$ the assumptions of the theorem hold and we have
\begin{align}
 \frac{1}{n} \sum_{i =1}^n \left( Y_{i+1} - \mathbb E[Y_{i+1}\mid\mathcal F_i]\right)\mathbbm{1}_{\{\tau >i\}} & = \\
 \frac{1}{n} \sum_{i =1}^n((Y_{i+1} - \mathbb E[Y_{i+1}\mid\mathcal F_i])&\mathbbm{1}_{\{\tau >i\}}-\mathbb E[(Y_{i+1} - \mathbb E[Y_{i+1}\mid\mathcal F_i])\mathbbm{1}_{\{\tau >i\}}\mid \mathcal F_i]) \xrightarrow{n \rightarrow \infty} 0 \notag
\end{align}
$\mathbb P$-almost surely.

Now observe that
\begin{align*}
 \liminf_{n \rightarrow \infty} \frac{1}{n} Y_{n\land \tau}& = \liminf_{n \rightarrow\infty} \frac{1}{n}\sum_{i =1}^n (Y_i - Y_{i-1})\mathbbm{1}_{\{\tau>i-1\}} \\
									& = \liminf_{n \rightarrow\infty} \frac{1}{n}\sum_{i =1}^n ((Y_i - Y_{i-1})\mathbbm{1}_{\{\tau>i-1\}} - \mathbb E[(Y_i - Y_{i-1})\mathbbm{1}_{\{\tau>i-1\}}\mid \mathcal F_{i-1}] \\
									 & \qquad + \mathbb E[(Y_i - Y_{i-1})\mathbbm{1}_{\{\tau>i-1\}}\mid \mathcal F_{i-1}])\\
									 & = \liminf_{n \rightarrow\infty} \frac{1}{n}\sum_{i =1}^n ((Y_i - \mathbb E[Y_i\mid \mathcal F_{i-1}])\mathbbm{1}_{\{\tau>i-1\}} \\
									 & \qquad + \mathbb E[(Y_i - Y_{i-1})\mathbbm{1}_{\{\tau>i-1\}}\mid \mathcal F_{i-1}])\\
									 & \hspace{-5pt}\overset{\eqref{eq:appHH}}{=} \liminf_{n \rightarrow\infty} \frac{1}{n}\sum_{i =1}^n \mathbb E[(Y_i - Y_{i-1})\mathbbm{1}_{\{\tau>i-1\}}\mid \mathcal F_{i-1}]\\
									 & = \liminf_{n \rightarrow\infty} \frac{1}{n}\sum_{i =1}^n \underbrace{\mathbb (\mathbb E[Y_i \mid \mathcal F_{i-1}]- Y_{i-1})}_{\geq A}\mathbbm{1}_{\{\tau>i-1\}} \\
									 & \geq \liminf_{n \rightarrow\infty}  \frac{n\land \tau}{n}A
\end{align*}
and therefore
\begin{align*}
 \liminf_{n \rightarrow \infty} \frac{1}{n} Y_n \geq A \text{ on } \{\tau=\infty\}
\end{align*}
which proves the first statement.

To prove the second statement we start by considering a process $(\bar Y_n)_{n \in \mathbb N}$ with the same properties as $(Y_n)_{n \in \mathbb N_0}$, but without the restriction on the size of the predecessor, i.e. such that for all $n \in \mathbb N_0$
 \begin{enumerate}
  \item[(1')] $\mathbb E[\bar Y_{n+1}\mid\mathcal F_n] \geq \bar Y_n + A$ and
  \item[(2')] $\mathbb E[(\bar Y_{n+1} -\mathbb E[\bar Y_{n+1} \mid \mathcal F_n])^2\mid \mathcal F_n] \leq B$
 \end{enumerate}
$\mathbb P$-almost surely.

With this define
\begin{align*}
 h_{i+1}:= \bar Y_{i+1} - \mathbb E[\bar Y_{i+1} \mid \mathcal F_i],\qquad
 S_i:=\sum_{j = 1}^i h_j
\end{align*}
for all $i \in \mathbb N_0$.
Note that due to (2') we know
\begin{align*}
 \mathbb E[h_i^2\mid \mathcal F_0] \leq B \text{ and } \mathbb E[h_ih_j\mid \mathcal F_0]=0
\end{align*}
holds for every $i,j \in \N_0, i\neq j$ $\mathbb P$-a.s.
For arbitrary constants $c_1 \geq c_2 \geq 0$ and $\alpha_1 > \alpha_2 >0$ we can then estimate
\begin{align*}
 & \mathbb P(\exists \, m \in \N: \; S_m \leq -c_1 -\alpha_1 m\mid \mathcal F_0) 	\\
 & \qquad \leq \W(\exists\,n \in \N:\; S_{n^2} \leq -c_2 -\alpha_2 n^2\mid \mathcal F_0) \\
 & \qquad \qquad + \W(\exists \, n \in \N\,\exists\, m \in [n^2, (n+1)^2-1]:\; S_m -S_{n^2} \leq -(c_1-c_2) -(\alpha_1 - \alpha_2)n^2\mid \mathcal F_0)\\
 & \qquad \leq \sum_{n \in \N} \W(S_{n^2} \leq -c_2 -\alpha_2 n^2\mid \mathcal F_0) \\
 & \qquad \qquad + \sum_{n\in \N} \sum_{m = n^2}^{(n+1)^2 -1} \underbrace{\W( S_m -S_{n^2} \leq -(c_1-c_2) -(\alpha_1 - \alpha_2)n^2\mid \mathcal F_0)}_{\leq \frac{\sum_{i=n^2+1}^{m}\mathbb E[ h_i^2\mid \mathcal F_0]}{((c_1-c_2) + (\alpha_1 -\alpha_2)n^2)^2}}\\
 & \qquad \leq \underbrace{\sum_{n \in \N} \frac{B n^2}{(c_2 + \alpha_2 n^2)^2}}_{=:f(c_2, \alpha_2)} + \underbrace{\sum_{n \in \N}\frac{B n(2n+1)}{((c_1-c_2) + (\alpha_1 -\alpha_2)n^2)^2}}_{=:g(c_1-c_2,\alpha_1-\alpha_2)}
\end{align*}
$\mathbb P$-a.s.,where $\lim_{c_2 \rightarrow \infty}f(c_2, \alpha_2) = 0$ and $\lim_{c \rightarrow \infty}g(c, \alpha_1-\alpha_2)=0$.
This means that for every $\theta >0$ (and every choice of $\alpha_1>\alpha_2>0$) choosing $c_2$ large enough for $f(c_2,\alpha_2) \leq \frac{\theta}{2}$ and then $c_1$ large enough  such that $g(c_1-c_2, \alpha_1-\alpha_2) < \frac{\theta}{2}$ we have
\begin{align*}
 \W(\exists \, m \in \N: \; S_m \leq -c_1 -\alpha_1 m\mid \mathcal F_0) 	& \leq f(c_2, \alpha_2) + g(c_1-c_2,\alpha_1-\alpha_2) \leq \theta
\end{align*}
$\mathbb P$-a.s.
Using $\alpha_1:= A$ we obtain the following for our process $(\bar Y_n)_{n \in \N_0}$: For every $\theta>0$ and every point $b$ choosing $S:= c_1+b$ for $c_1$ as above we get that on $\{\bar Y_0\geq S\}$
\begin{align*}
 \W(\exists\, m \in \N:\; \bar Y_m \leq b\mid \mathcal F_0 )	& = \W(\exists\, m \in \N:\; \bar Y_0 + \sum_{i=1}^m (\bar Y_i-\bar Y_{i-1})  \leq b \mid \mathcal F_0)\\
						& \leq \W(\exists\, m \in \N:\; S + \sum_{i=1}^m (\bar Y_i-\bar Y_{i-1})  \leq b \mid \mathcal F_0) \\
						& \leq \W(\exists\, m \in \N:\; \sum_{i=1}^m (\bar Y_i-\bar Y_{i-1}) \leq b - S   \mid \mathcal F_0)\\
						& \leq \W(\exists\, m \in \N:\; \sum_{i=1}^m h_i \leq b-S - \sum_{i=1}^m \underbrace{\mathbb E[\bar Y_i-\bar Y_{i-1}\mid \mathcal F_{i-1}]}_{\geq A \text{ by (1')}}  \mid \mathcal F_0) \\
						& \leq \W(\exists\, m \in \N:\; \sum_{i=1}^m h_i \leq \underbrace{b-S}_{=-c_1} - m\underbrace{A}_{=\alpha_1} \mid \mathcal F_0) \leq \theta,
\end{align*}
which means that for every $\theta >0$ and $b \in \mathbb R$ we can find an $S$ such that
\begin{align*}
 \W(\forall \, m \in \N:\; \bar Y_m > b \mid \mathcal F_0)\geq 1-\theta \;\text{ on } \{\bar Y_0\geq S\}.
\end{align*}

Coming back to $(Y_n)_{n \in \N_0}$ use it to define such a process $(\bar Y_n)_{n \in \N_0}$ through $\bar Y_0:= Y_0$ and
\begin{align*}
 \bar Y_{n+1} -\bar Y_n :=\begin{cases}
                           Y_{n+1} -Y_n,	& \text{if } \tau >n\\
                           A,			& \text{otherwise.}
                          \end{cases}
\end{align*}
This process has the stronger properties (1') and (2') and since we also have $\{ \forall n\ \in \N_0: \bar Y_n \geq a\} = \{\forall n \in \N_0: Y_n \geq a\}$ the above observation yields the second statement.
\end{proof}

\vskip 0.6 truecm

{\bf Acknowledgments} The first--named author (U.U.J.) thanks the IMU Berlin Einstein Foundation Program (EFP) and the Berlin Mathematical School (BMS)
for a scholarship and for supporting his visit to the  Technische  Universit\"at (TU) Berlin, Germany and the Programme Erasmus Mundus Action 2 (EMA2) Marco XXI for a scholarship and for supporting his visit to the University of Santiago de Compostela (USC), Spain. He also thanks the TU Berlin and USC for the kind hospitality and for  providing all facilities.


\begin{thebibliography}{99}
\bibitem{Ar98}
{\it Arnold, L. \/} Random Dynamical Systems. {\it Springer}, 1998.
\bibitem{Br}
{\it Bernstein S.N. \/} The solution of a mathematical
problem related to the theory of heredity.  {\it Uchn. Zapiski. NI
Kaf. Ukr. Otd. Mat.} No. 1 (1924), 83--115 (in Russian).
\bibitem{ChuScheu04}
{\it Chueshov, I. and Scheutzow, M. \/} On the structure of attractors and invariant measures for a class of monotone random systems.
{\it Dyn. Syst.} {\bf 19} (2004), 127--144.
\bibitem{Crauel99}
{\it Crauel, H. \/} Global random attractors are uniquely determined by attracting deterministic compact sets.
{\it Ann. Math. Pura Appl.} {\bf 176} (1999), 57--72.
\bibitem{Crauel01}
{\it Crauel, H. \/} Random point attractors versus random set attractors.
{\it J. London Math. Soc.} {\bf 63} (2001), 413--427.
\bibitem{CDS09}
{\it Crauel, H., Dimitroff, G. and Scheutzow, M. \/} Criteria for strong and weak random attractors.
{\it J. Dynam. Differential Equations} {\bf 21} (2009), 233--247.
\bibitem{CF94}
{\it Crauel, H. and Flandoli, F. \/} Attractors for random dynamical systems.
{\it Probab. Theory Related Fields} {\bf 100} (1994), 365--393.
\bibitem{D}
{\it Devaney R.L. \/}  An Introduction to Chaotic Dynamical Systems.
{\it Westview Press}, 2003.
\bibitem{E}
{\it Elaydi S.N.} Discrete Chaos. {\it Chapman Hall/CRC}, 2000.
\bibitem{NN0}
{\it Ganikhodzhaev N. N. \/} An application of the theory of Gibbs distributions to mathematical genetics,
{\it Doklady Math.},  {\bf 61}:3 (2000), 321--323.
\bibitem{NN}
{\it Ganikhodzhaev N. N. \/} The random models of heredity in random environment,
{\it Dokl. Acad. Nauk RUz} No.12 (2001), 6--8.
\bibitem{GZ1}
{\it Ganikhodjaev N.N., Zanin D.V.\/} On a necessary condition for the
ergodicity of quadratic operators defined on the two-dimensional
simplex,{\it Russ. Math.Surv.} {\bf59}:3 (2004), 571-572.
\bibitem{GZ2}
{\it Ganikhodjaev N.N., Zanin D.V. \/} Ergodic Volterra quadratic maps
of the simplex,  \\ \  \ {\it arXiv:} \  \  1205.3841. (Russian)
\bibitem{GJM1}
{\it Ganikhodjaev N.N., Jamilov(Zhamilov) U.U. and  Mukhitdinov R.T. \/} On Non-Ergodic Transformations on $S^3,$
Journal of Physics: Conference Series {\bf 435} (2013) 012005.
\bibitem{GGJ}
{\it Ganikhodjaev N.N., Ganikhodjaev R.N. and Jamilov U.U. \/} Quadratic stochastic operators and zero-sum game dynamics,
{\it Ergod. Th. and Dynam. Sys.} ({\it to appear}).
\bibitem{GJM2}
{\it Ganikhodjaev N.N., Jamilov U.U. and  Mukhitdinov R.T. \/} Non-ergodic quadratic operators of bisexual population,
{\it Ukr. Math. Jour.} {\bf 65}:6 (2013), 1152-1160.
\bibitem{RN1}
{\it Ganikhodzhaev R. N. \/} Quadratic stochastic operators, Lyapunov functions, and
tournaments, {\it Sb.Math.} {\bf 76}:2  (1993), 489--506.
\bibitem{RN2}
{\it Ganikhodzhaev R. N.\/}  Map of fixed points and Lyapunov functions for a class of discrete dynamical systems,
{\it Math. Notes} {\bf 56}:5 (1994), 1125--1131 .
\bibitem{RNEs}
{\it Ganikhodzhaev R. N. and Eshmamatova D. B. \/} Quadratic automorphisms of a simplex and the asymptotic
behavior of their trajectories, {\it Vladikavkaz. Mat. Zh.} {\bf 8}:2 (2006),  12--28, (in Russian).
\bibitem{GMR}
{\it Ganikhodzhaev R.N., Mukhamedov F.M., Rozikov U.A. \/} Quadratic
stochastic operators and processes: Results and open problems. {\it Inf. Dim. Anal., Quantum Prob.
and Rel. Top.}. Vol. {\bf 14}:2 (2011), 279--335.
\bibitem{HH}
{\it Hall P. and Heyde C.C.} Martingale Limit Theory and its Application.  {\it Academic Press} (1980).
\bibitem{K1}
{\it Kesten H. \/} Quadratic transformations: a model for population growth. I, {\it Adv. Appl.
Prob.} {\bf 2}:1 (1970), 1--82.
\bibitem{K2}
{\it Kesten H. \/}  Quadratic transformations: a model for population growth. II, {\it Adv. Appl.
Prob.} {\bf 2}:2 (1970), 179--228.
\bibitem{Krengel}
{\it Krengel, U.\/} Einf\"uhrung in die Wahrscheinlichkeitstheorie und Statistik, 6. Auflage// {\it Vieweg Studium: Aufbaukurs Mathematik [Vieweg Studies:
              Mathematics Course]} (2002).
\bibitem{Lyu1}
{\it Lyubich Yu.I. \/} Mathematical structures in population genetics.
{\it Biomathematics}, {\bf 22}, Springer-Verlag, (1992).
\bibitem{MAT}
{\it Mukhamedov F., Akin H. and Temir S.} On infinite dimensional quadratic Volterra
operators, {\it J. Math. Anal. Appl}. {\bf 310} (2005), 533--556.
\bibitem{R}
{\it Robinson R.C.}  An Introduction to Dynamical Systems:
Continuous and Discrete. {\it Pearson Education}, 2004.
\bibitem{RJ1}
{\it Rozikov U.A. and Jamilov U.U. \/} F-quadratic stochastic operators,
{\it Math. Notes} {\bf 83}:4 (2008), 554--559.
\bibitem{RJ2}
{\it Rozikov U.A. and Jamilov U.U. \/} The dynamics of strictly non-Volterra
quadratic stochastic operators on the two deminsional simplex,
{\it Sb.Math.} {\bf 200}:9 (2009), 1339--1351.
\bibitem{RJ3}
{\it Rozikov U.A. and Jamilov U.U. \/} Volterra quadratic stochastic of a two-sex population,
{\it Ukr.Math. Jour.} {\bf 63}:7 (2011), 1136--1153.
\bibitem{Scheu02}
{\it Scheutzow, M. \/} Comparison of various concepts of a random attractor: A case study.
{\it Archiv d. Math.} {\bf 78} (2002), 233--240.
\bibitem{ScheuStein}
{\it Scheutzow, M. and Steinsaltz, D. \/} Chasing balls through martingale fields.
{\it Ann. Probab.} {\bf 30} (2002), 2046--2080.
\bibitem{Ul}
{\it Ulam S.M. \/} A collection of mathematical problems//  New
York�London.:Interscience Publ., (1960).
\bibitem{Val}
{\it Vallander S.S. \/} On the limit behavior of iteration sequence of certain quadratic
transformations //{\it Soviet Math.Doklady} {\bf 13} (1972), 123--126.
\bibitem{Zakh}
{\it Zakharevich M.I.\/} On the behavior of trajectories and the ergodic hypothesis for quadratic mappings of a simplex// {\it Russ. Math. Surv.} {\bf 33}:6  (1978), 265--266.

\end{thebibliography}
\end{document}